\documentclass[a4paper,11pt]{amsart}

\usepackage{amsmath,amsthm,amsfonts,amssymb,url,tikz}
\usepackage[margin=1in]{geometry}

\title{A note on growth of hyperbolic groups}
\author{Motiejus Valiunas}
\address{Mathematical Sciences, University of Southampton, Southampton SO17 1BJ, United Kingdom}
\email{m.valiunas@soton.ac.uk}

\newcommand{\bb}{\mathfrak{B}}
\newcommand{\sss}{\mathfrak{S}}

\theoremstyle{plain}
\newtheorem{thm}{Theorem}
\newtheorem{lem}[thm]{Lemma}
\theoremstyle{definition}
\newtheorem*{ack}{Acknowledgement}

\begin{document}

\begin{abstract}
The following short note provides an alternative proof of a result of Coornaert \cite{coornaert}: namely, that given a non-elementary word-hyperbolic group $G$ with a finite generating set $X$, there exist constants $\lambda,D > 1$ such that \[ D^{-1}\lambda^n \leq |B_{G,X}(n)| \leq D \lambda^n \] for all $n \geq 0$, where $B_{G,X}(n)$ is the ball of radius $n$ in the Cayley graph $\Gamma(G,X)$.
\end{abstract}
\maketitle

Given a group $G$ with a finite generating set $X$ and an integer $n \geq 0$, we denote by $S_{G,X}(n) \subseteq G$ the \emph{sphere} of radius $n$ in $G$ with respect to $X$: that is, the set of elements $g \in G$ that can be represented by words of $n$ (but not $n-1$) letters in $X^{\pm 1}$. We denote $|S_{G,X}(n)|$ by $\sss_{G,X}(n)$. Similarly, we denote by $B_{G,X}(n) \subseteq G$ the \emph{ball} of radius $n$ in $G$ with respect to $X$ -- that is, $B_{G,X}(n) = \bigcup_{i=0}^n S_{G,X}(i)$ -- and we write $\bb_{G,X}(n)$ for $|B_{G,X}(n)|$. We write $S(n)$, $\sss(n)$, $B(n)$ and $\bb(n)$ for $S_{G,X}(n)$, $\sss_{G,X}(n)$, $B_{G,X}(n)$ and $\bb_{G,X}(n)$ (respectively) if $G$ and $X$ are clear.

Here we give a short alternative proof of the following result.
\begin{thm}[Coornaert {\cite[Th\'eor\`eme 7.2]{coornaert}}] \label{t:coornaert}
Let $G$ be a non-elementary word-hyperbolic group with a finite generating set $X$. Then there exist constants $\lambda,D > 1$ such that \[ D^{-1}\lambda^n \leq \bb_{G,X}(n) \leq D \lambda^n \] for all $n \geq 0$.
\end{thm}
The proof of this theorem given in \cite{coornaert} is based on the theory of Patterson-Sullivan measures and dynamical systems; using similar methods, Yang has shown an analogous result for relatively hyperbolic groups \cite[Theorem 1.9]{yangrh}. Here we give a Patterson-Sullivan measure-free proof. It is worth noting that an analogous result also holds when $G$ is a right-angled Artin/Coxeter group that does not split as a direct product and $X$ is the standard generating set \cite[Theorem 2.2]{gtt}.

We define the (\emph{spherical}) \emph{growth function} of $G$ with respect to $X$ to be the formal power series $s_{G,X}(t) = \sum_{n=0}^\infty \sss_{G,X}(n)t^n$. Similarly, we may define the \emph{volume growth function} as $b_{G,X}(t) = \sum_{n=0}^\infty \bb_{G,X}(n)t^n$; note that we have $s_{G,X}(t) = (1-t)b_{G,X}(t)$. We say that $G$ has \emph{rational growth} with respect to $X$ if $s_{G,X}(t)$ (equivalently, $b_{G,X}(t)$) is a rational function: that is, a ratio of two polynomials. A classical result due to Cannon says that this is always the case for hyperbolic groups.
\begin{thm}[Cannon {\cite[Theorem 7]{cannon}}] \label{t:gromov}
A word-hyperbolic group $G$ has rational growth with respect to any finite generating set.
\end{thm}

Our proof relies on the following theorem of the author.
\begin{thm}[Valiunas {\cite[Theorem 1]{me1}}] \label{t:me}
Let $G$ be an infinite group with a finite generating set $X$, and suppose that $G$ has rational growth with respect to $X$. Then there exist constants $\alpha \in \mathbb{Z}_{\geq 0}$, $\lambda \in [1,\infty)$ and $D > C > 0$ such that \[ C n^\alpha \lambda^n \leq \sss_{G,X}(n) \leq D n^\alpha \lambda^n \] for all $n \geq 1$.
\end{thm}

For the rest of the note, fix a non-elementary word-hyperbolic group $G$ and a finite generating set $X$ for $G$. Let $\delta \in \mathbb{Z}_{\geq 1}$ be a constant such that triangles in the Cayley graph $\Gamma(G,X)$ are $\delta$-thin, and let $d$ be the combinatorial metric on $\Gamma(G,X)$. We use an auxiliary lemma.
\begin{lem} \label{l:me}
For all $n,m \geq 0$, we have $\sss(n) \sss(m) \leq \sum_{\ell=0}^{n+m} \sss(\ell) \bb\left(\delta+\left\lceil\frac{n+m-\ell}{2}\right\rceil\right)$.
\end{lem}

\begin{proof}
Let $n,m \geq 0$. The map
\begin{align*}
\mu: S(n) \times S(m) &\to B(n+m), \\ (h,k) &\mapsto hk,
\end{align*}
allows us to define a partition
\[
S(n) \times S(m) = \bigsqcup_{\ell=|n-m|}^{n+m} \mu^{-1}(S(\ell)).
\]
It is therefore enough to show that $|\mu^{-1}(g)| \leq \bb\left(\delta+\left\lceil\frac{n+m-\ell}{2}\right\rceil\right)$ for all $\ell$ and all $g \in S(\ell)$.

Let $\ell \in \{ |n-m|,\ldots,n+m \}$ and let $g \in S(\ell)$. Let $\gamma_g$ be a geodesic in $\Gamma(G,X)$ joining $1_G$ and $g$, and let $p_g$ be the point on $\gamma_g$ such that $d(p_g,1_G) = \frac{\ell+n-m}{2}$. Let $c \in G$ be a vertex of $\Gamma(G,X)$ such that $d(c,p_g) \leq \frac{1}{2}$, and note that $c=p_g$ if $n+m-\ell$ is even. Let $N = \delta+\left\lceil\frac{n+m-\ell}{2}\right\rceil$.

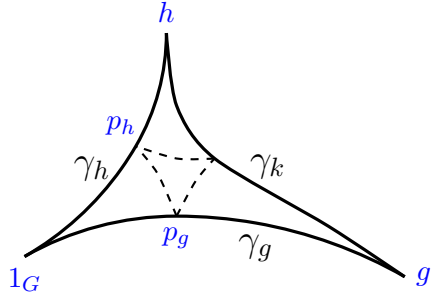
\begin{center}
\begin{tikzpicture}
\draw [very thick] (0,0) node [below,blue] {$1_G$} arc (120:90:4) node [below,blue] {$p_g$} arc (90:80:6) node [below] {\Large $\gamma_g$} arc (80:60:6) node [right,blue] {$g$};
\draw [very thick] (0,0) arc (-60:-35:4) node [left] {\Large $\gamma_h$} arc (-35:-30:4) node [above,blue] {$p_h\ \ \ $} arc (-30:0:3) node [above,blue] {$h$};
\draw [very thick] plot [smooth] coordinates { (1.8660,2.9641) (2,2) (2.5,1.3) (4,0.4) (5,-0.2679) };
\draw [thick,dashed] plot [smooth] coordinates { (2,0.5359) (1.75,1.1) (1.4641,1.4641) };
\draw [thick,dashed] plot [smooth] coordinates { (2,0.5359) (2.25,1) (2.5,1.3) };
\draw [thick,dashed] plot [smooth] coordinates { (2.5,1.3) (2,1.3) (1.4641,1.4641) };
\node at (3.2,1.2) {\Large $\gamma_k$};
\end{tikzpicture}
\end{center}

Let $(h,k) \in \mu^{-1}(g)$, so that $h \in S(n)$ and $k = h^{-1}g \in S(m)$. Let $\gamma_h$, $\gamma_k$ be geodesics in $\Gamma(G,X)$ joining $1_G$ and $h$, $h$ and $g$, respectively, and let $\Delta$ be the geodesic triangle with edges $\gamma_g$, $\gamma_h$ and $\gamma_k$. Let $p_h$ be the point on $\gamma_h$ such that $d(p_h,1_G) = \frac{\ell+n-m}{2}$. Since $\Delta$ is $\delta$-thin, we have $d(p_g,p_h) \leq \delta$. By construction, $d(p_h,h) = \frac{n+m-\ell}{2}$, and so it follows by the triangle inequality that $d(1_G,c^{-1}h) = d(c,h) \leq N$. In particular,
\[
|\mu^{-1}(g)| = |\{ h \in S(n) \mid h^{-1}g \in S(m) \}| \leq |\{ h \in S(n) \mid c^{-1}h \in B(N) \}| \leq |cB(N)| = \bb(N),
\]
as required.
\end{proof}

\begin{proof}[Proof of Theorem \ref{t:coornaert}]
By Theorem \ref{t:gromov}, we may apply Theorem \ref{t:me} to $G$ and $X$. Let $\lambda$, $\alpha$, $C$ and $D$ be as given by Theorem \ref{t:me}. Since $G$ is non-elementary, it has exponential growth and therefore $\lambda > 1$; without loss of generality, assume furthermore that $D \geq 1$. It then follows from an easy computation that
\begin{equation} \label{e:balls}
C n^\alpha \lambda^n < \bb(n) < \frac{D\lambda}{\lambda-1} n^\alpha \lambda^n
\end{equation}
for all $n \geq 1$. It is therefore enough to show that $\alpha = 0$.

Pick an integer $n \geq 1$. It follows from Theorem \ref{t:me}, Lemma \ref{l:me} and \eqref{e:balls} that
\[
C^2 n^{2\alpha} \lambda^{2n} \leq \sss(n)^2 \leq \sum_{\ell=0}^{2n} \sss(\ell) \bb(\delta+\lceil n-\ell/2 \rceil) < \frac{D^2\lambda}{\lambda-1} \sum_{\ell=0}^{2n} \ell^\alpha (\delta+\lceil n-\ell/2 \rceil)^\alpha \lambda^{\delta+n+\lceil \ell/2 \rceil}.
\]
We may rearrange this to obtain
\begin{equation} \label{e:0}
\frac{C^2 (\lambda-1)}{2^\alpha D^2 \lambda} < \sum_{\ell=0}^{2n} \underbrace{\left( \frac{\ell}{2n} \right)^\alpha}_{\leq 1} \underbrace{\left( \frac{\delta+\lceil n-\ell/2 \rceil}{n} \right)^\alpha}_{\leq (1+\delta/n)^\alpha} \lambda^{\delta+\lceil \ell/2 \rceil-n}.
\end{equation}

Now let $\varepsilon: \mathbb{Z}_{\geq 1} \to \mathbb{Z}_{\geq 2\delta}$ be a function such that $n\left(1+\frac{\delta}{n}\right)^\alpha\lambda^{-\varepsilon(n)} \to 0$ and $\frac{\varepsilon(n)^2}{n} \to 0$ as $n \to \infty$: for instance, $\varepsilon(n) = 2\delta+\lfloor \sqrt[3]{n} \rfloor$. We may split the right hand side of \eqref{e:0} into terms with $\ell \leq 2n-2\varepsilon(n)$ and terms with $\ell > 2n-2\varepsilon(n)$. In particular, we get
\begin{equation} \label{e:1}
\begin{aligned}
\sum_{\ell=0}^{2n-2\varepsilon(n)} \left( \frac{\ell}{2n} \right)^\alpha \left( \frac{\delta+\lceil n-\ell/2 \rceil}{n} \right)^\alpha \lambda^{\delta+\lceil \ell/2 \rceil-n} &\leq \sum_{\ell=0}^{2n-2\varepsilon(n)} \left(1+\frac{\delta}{n}\right)^\alpha \lambda^{\delta-\varepsilon(n)} \\ &< 2n \left(1+\frac{\delta}{n}\right)^\alpha \lambda^{\delta-\varepsilon(n)}
\end{aligned}
\end{equation}
and
\begin{equation} \label{e:2}
\begin{aligned}
\sum_{\ell=2n-2\varepsilon(n)+1}^{2n} \left( \frac{\ell}{2n} \right)^\alpha \left( \frac{\delta+\lceil n-\ell/2 \rceil}{n} \right)^\alpha \lambda^{\delta+\lceil \ell/2 \rceil-n} &\leq \sum_{\ell=2n-2\varepsilon(n)+1}^{2n} \left( \frac{\delta+\varepsilon(n)}{n} \right)^\alpha \lambda^{\delta} \\ &= 2\lambda^\delta \varepsilon(n) \left( \frac{\delta+\varepsilon(n)}{n} \right)^\alpha.
\end{aligned}
\end{equation}

Note that the left hand side of \eqref{e:0} is a strictly positive constant, and the right hand side of \eqref{e:1} tends to zero as $n \to \infty$ by the choice of $\varepsilon$. Furthermore, by the choice of $\varepsilon$, if $\alpha \geq 1$ then the right hand side of \eqref{e:2} tends to zero as $n \to \infty$. But by \eqref{e:0} this cannot happen, and so $\alpha = 0$, as required.
\end{proof}

\begin{ack} I would like to thank Laura Ciobanu for a conversation that inspired this note and for a subsequent discussion. \end{ack}

\bibliographystyle{amsalpha}
\bibliography{../../all}

\end{document}